\documentclass[11pt,a4paper]{amsart}

\title[A smoothness criterion for complex spaces]{A smoothness criterion for complex spaces\\ in terms of differential forms}

\author{H{\aa}kan Samuelsson Kalm \& Martin Sera}
\address{H{\aa}kan Samuelsson Kalm, Martin Sera, Department of Mathematical Sciences, Division of Algebra and Geometry, University of Gothenburg and 
Chalmers University of Technology, SE-412 96 G\"{o}teborg, Sweden}
\email{hasam@chalmers.se}
\address{Martin Sera, Faculty of Engineering, Kyoto University of Advanced Science, Kyoto 615-8577, Japan}
\email{sera.martin@kuas.ac.jp}


\date{\today}

\usepackage[english]{babel}
\usepackage{amsmath,calligra}
\usepackage{amsthm,amssymb,latexsym}
\usepackage[all,cmtip]{xy}
\usepackage{url,ae}
\usepackage{t1enc}
\usepackage{mathrsfs}
\usepackage{graphicx}
\usepackage{geometry}
\usepackage[breaklinks=true, colorlinks=false]{hyperref} 

\usepackage{setspace}

\hypersetup{
  pdftoolbar=true,
  pdfmenubar=true,
  pdftitle={A smoothness criterion for complex spaces in terms of differential forms},
  pdfauthor={H. Samuelsson Kalm and M. Sera},
  pdfsubject={Preprint},
  pdfkeywords={},
pdfborder= 0 0 .01,
  bookmarksnumbered=false,
}

\newtheorem{proposition}{Proposition}[section]
\newtheorem{theorem}[proposition]{Theorem}

\newtheorem{corollary}[proposition]{Corollary}

\theoremstyle{definition}

\numberwithin{equation}{section}

\DeclareMathOperator{\Hom}{\mathscr{H}\text{\kern -3pt {\calligra\Large om}}\,}
\DeclareMathOperator{\Ext}{\mathscr{E}\text{\kern -3pt {\calligra\Large xt}}\,\,}
\DeclareMathOperator{\Image}{\mathscr{I}\text{\kern -3pt {\calligra\Large m}}\,}
\DeclareMathOperator{\Ker}{\mathscr{K}\text{\kern -3pt {\calligra\Large er}}\,}
\newcommand{\CC}{\mathbb{C}}
\newcommand{\debar}{\bar{\partial}}

\newcommand{\PM}{\mathscr{P} \kern -3pt \mathscr{M}}

\newcommand{\hol}{\mathscr{O}}

\newcommand{\CH}{\mathscr{C} \kern -2pt \mathscr{H}}

\newcommand{\Om}{\mathit{\Omega}} 
\newcommand{\tOm}{{\tilde{\Om\,}\!}{}}

\def\newop#1{\expandafter\def\csname #1\endcsname{\mathop{\rm #1}\nolimits}}
\newop{span}

   \newcommand{\sE}{\mathscr{E}} \newcommand{\sF}{\mathscr{F}}       \newcommand{\sO}{\mathscr{O}}   \newcommand{\sS}{\mathscr{S}} \newcommand{\sT}{\mathscr{T}}   

\newcommand{\aus}{\subset} \newcommand{\minus}{\setminus} \newcommand{\iso}{\simeq}

\newcommand{\ph}{\varphi}

\newcommand{\reg}{{\mathrm{reg}}}\newcommand{\sing}{{\mathrm{sing}}}\newcommand{\tor}{{\mathrm{torsion}}}

\DeclareMathOperator{\codim}{codim}

\newcommand{\cf}{cf.\ } \newcommand{\eg}{e.\,g.,\ }   \newcommand{\ie}{i.\,e., }  
\newcommand{\spfrac}[2]{#1/#2}

\makeatletter
\def\section{\@startsection{section}{1}
  \z@{2\linespacing\@plus1.5\linespacing}{\linespacing\@plus .5\linespacing}
  {\normalfont\bf\centering}}
\makeatother

\newcommand{\stdpt}[1]{\makebox[1.75em]{(#1)}}

\begin{document}
\nocite{*}
\bibliographystyle{plain}

\begin{abstract}
For a reduced pure dimensional complex space $X$, we show that if 
Barlet's recently introduced sheaf $\alpha_X^1$ of holomorphic $1$-forms or
the sheaf of germs of weakly holomorphic $1$-forms is locally free, then $X$ is smooth.
Moreover, we discuss the connection to Barlet's well-known sheaf $\omega_X^1$.
\end{abstract}

\thanks{The last author was supported by the German Research Foundation\\ (DFG, grant SE 2677/1) and the Knut and Alice Wallenberg Foundation.}

\maketitle

\onehalfspacing
\thispagestyle{empty} 

\section{Introduction}
Let $X$ be a reduced pure dimensional complex space.
In presence of singularities, there exist several natural concepts of holomorphic differential forms on $X$.
Beside of the well-known sheaf $\Om_X^1$ of K\"ahler differentials, the sheaf $\Om_X^1/\tor$ of strongly holomorphic forms, and the so-called Barlet sheaf $\omega_X^1$,
we will consider the following two in this note.
We say that a holomorphic $1$-form $\varphi$ on $X_{\reg}$ is \emph{weakly holomorphic} if 
there is a resolution of singularities
$\pi\colon\tilde{X}\to X$, with $\tilde{X}$ smooth, such that $\pi^*\varphi$ extends holomorphically across the exceptional set. 
By \cite[Section II]{Griffiths},
$\varphi$ is weakly holomorphic if and only if for any holomorphic curve $C\subset X$ not contained in $X_{\sing}$,
$\varphi\wedge\bar{\varphi}$ is integrable on the regular part of any small enough open subset in $C$.
We will denote the sheaf of germs of weakly 
holomorphic $1$-forms by $\tOm_X^1$;
thus $\tOm_X^1=\pi_*\Om_{\tilde{X}}^1$.
In \cite[Section 3]{Barlet18}, Barlet introduces the coherent analytic subsheaf $\alpha_X^1$ of $\tOm_X^1$ which is defined as follows.
Using \cite[Theorem 3.5]{Rossi} (\cf also \cite[\S\,2]{Riemenschneider}, \cite[Prop.~2.1.1]{Barlet18}),
we may assume that $\pi$ is chosen such that $\pi^\ast\Om_X^1/\tor$ is locally free.
On a small enough open set,
let $\ph_1,...,\ph_N$ be holomorphic $1$-forms generating $\Om_{X}^1$.
Then, $\pi^{\ast\ast}\Om^1_X$ is the $\hol_{\tilde X}$-module sheaf on $\tilde X$ generated by $\pi^\ast\ph_1,...,\pi^\ast \ph_N$ modulo torsion, where $\pi^\ast \ph_k$ denotes the pullback as differential form.
$\alpha_X^1$ is defined as the direct image sheaf $\pi_\ast(\pi^{\ast\ast}\Om_X^1)$.
We have $\Om_X^1/\tor\subset \alpha_X^1\subset \tOm_X^1$ which are proper inclusions in general.
Furthermore,  there is a pullback functor for $\alpha_X^1$ (compatible with the pullback of strongly holomorphic forms) which cannot exist for weakly holomorphic forms.
The purpose of this note is to show 

\begin{theorem}\label{main}
\stdpt{i} $X$ is smooth if and only if $\tOm_X^1$ is a locally free $\hol_X$-module.\\
\stdpt{ii} $X$ is smooth if and only if $\alpha_X^1$ is a locally free $\hol_X$-module.
\end{theorem}

The only if-parts are of course trivial. Our contribution is essentially the observation that if $\tOm_X^1$ or $\alpha_X^1$ is locally free, then $X$ 
is normal and the tangent sheaf, $\mathscr{T}_X$, 
is locally free. That $X$ then is smooth follows by an argument of van Straten and Steenbrink, \cite[Section 1.6]{vSS}, elaborated by Greb, 
Kebekus, Kov\'{a}cs, and Peternell, \cite[Theorem 6.1]{GKKP} which we outline for completeness (see Proposition \ref{pro:general_GKKP}).
The (i)-part was already proven by Kersken with a different approach in \cite[Satz~3.1]{Kersken}.

\medskip

Obviously, Theorem~\ref{main} generalizes  the classical smoothness criterion: if the sheaf of K\"ahler differentials $\Om_X^1$ or the strongly holomorphic 1-forms $\Om_X^1/\tor$ is locally free, then $X$ is smooth.
It is a well-known problem whether there is a smoothness criterion in terms of the sheaf $\omega_X^1$ introduced by Barlet, \cite[Section 1]{Barlet}.
$\omega_X^1$ may be defined as the sheaf of $\debar$-closed currents  on $X$ modulo the $\debar$-closed currents with support in the singular set $X_\sing$, \cf \cite[Prop.~4]{Barlet}.
Sections of $\omega_X^1$ over $U\setminus A$, where $U\subset X$ is open and $A\subset U$ is an analytic set with $\text{codim}_U A\geq 2$, extend across $A$ and, moreover,
$\omega_X^1$ is always torsion free, see \cite[Section 1]{Barlet}. It follows that $\omega_X^1$ is reflexive if $X$ is normal, see \cite[Section 2]{Barlet}.
In view of Theorem~\ref{main}, a smoothness criterion in terms of $\omega^1_X$ follows for spaces $X$ such that $\tOm_X^1 \simeq \omega_X^1$.
In general, $\tOm_X^1 \subsetneq \omega_X^1$ and the quotient 
$\omega_X^1/\tOm_X^1$ is related to the $s^{(1)}$-invariant of isolated singularities, see \cite{Yau}. However, Flenner's main result in \cite{Flenner}
implies that $\tOm_X^1$ is reflexive if $X$ is normal and $\text{codim}_X X_{\sing} \geq 3$.
If $X$ has klt singularities, then $\tOm_X^1$ is reflexive as well, see \cite[Theorem 1.4]{GKKP}.
In both cases, we obtain that $\tOm_X^1 =\omega_X^1$
since on a normal space a reflexive sheaf is uniquely determined by its restriction to the regular part.
Moreover, Pinkham and Wahl have shown that this equality also holds if $X$ is a surface with rational singularities, see, see 
\cite[Appendix]{Pinkham}  (cf.\ also \cite[\S\- 4--5]{Wahl}).
We remark that this last result does not hold in general in positive characteristic. 
Inspired by Pinkham, \cite[Prop.~1, Appendix]{Pinkham}, we get a sufficient cohomological condition ensuring the equality
$\tOm_X^1 =\omega_X^1$, see Proposition \ref{pro:pinkham} below, and we conclude the following corollary of Theorem~\ref{main}.

\begin{corollary}\label{cor:pinkham}
Let $X$ be a reduced Stein space and $M\to X$ a resolution of singularities with exceptional divisor $E$ having normal crossings.
If $\omega_X^1$ is locally free and
$H_E^1(\Om_M^1(\log E)\otimes\sO(-E))=0$,
then $X$ is smooth.
\end{corollary}

In Corollary \ref{cor:for_the_referee} below, we collect all cases discussed in this note where the local freeness of $\omega_X^1$ implies $X$ is smooth.

\medskip

Let us put this in the context of the Lipman-Zariski conjecture, \cite[Introduction]{Lipman}, 
which states that $X$ is smooth if and only if $\mathscr{T}_X$ is locally free. 
Lipman \cite[Theorem 1]{Lipman} showed that if $\mathscr{T}_X$ is locally free, then $X$ is at least normal.
The conjecture has been proved assuming a priori, \eg that $\text{codim}_X X_{\sing}\geq 3$ or that $X$ is a klt space, see
\cite[Section 1.6]{vSS}, \cite[Corollary]{Flenner}, \cite[Theorem 6.1]{GKKP}, and the references therein. 
On a normal space the Lipman-Zariski conjecture may be reformulated in terms of $\omega_X^1$.
Since the dual sheaf $\mathscr{T}_X^*$ and $\omega_X^1$ are reflexive, $\mathscr{T}_X^*\simeq \omega_X^1$ if $X$ is normal.
Therefore, since $\sT_X$ is reflexive (as the dual of the K\"ahler differentials), on a normal space $X$,  
the Lipman-Zariski conjecture is equivalent to the statement that
$X$ is smooth if and only if $\omega_X^1$ is locally free.
We conclude that the Lipman-Zariski conjecture holds for complex spaces satisfying the cohomology condition in Corollary \ref{cor:pinkham}, in particular, for surfaces with rational singularities (the latter was proved also by Kersken, see \cite[(3.3)]{Kersken}).

\bigskip

\noindent
{\bf Acknowledgment:} We would like to thank Jan Stevens for fruitful discussions and suggestions and the anonymous referee for valuable comments which helped to improve the paper.

\section{Proofs}
A crucial ingredient of the proof of Theorem \ref{main} is the following proposition which generalizes a result by the second author in \cite[Sect.~4]{Sera16}:

\begin{proposition}\label{normality}
If $\sE$ is a locally free sheaf of positive rank on a reduced complex space $X$ such that there exist a proper modification $\pi\colon Z\rightarrow X$ with $Z$ normal and a coherent analytic sheaf $\sF$ with $\pi_\ast \sF\iso \sE$, then $X$ is normal.
\end{proposition}
\begin{proof}
We only need to prove that $X$ is locally irreducible. That $X$ is normal then follows from \cite[Theorem 4.6]{Sera16}.

Let $x\in X_\sing$ and assume to get a contradiction that $X$ is reducible at $x$.
Then, there is a connected neighbourhood $U$ of $x$ and a decomposition of $U$ in irreducible components $U_1{,}...,U_m$, $m\geq 2$. Since $Z$ is normal, the preimages $V_i=\pi^{-1}(U_i)$, $i=1{,}...,m$ are pairwise disjoint. Set $V:=\pi^{-1}(U)$.
We may assume that $\hol_U^r\iso\sE_U\iso \pi_\ast\sF_V$.
Let $\Phi\colon \hol_U^r\rightarrow \pi_\ast\sF_V$ be the composition of these isomorphisms.
Consider $e_1:=(1\ 0\ ...\ 0)^T\in\hol^r(U)$ and $g:=\Phi(e_1)\in \pi_\ast\sF(U)=\sF(V)$.
We denote the restriction of $g$ to $V_i$ by $\tilde g_i\in \sF(V_i)$.
Let $g_i\in \sF(V)$ be the trivial extension of $\tilde g_i$, which is well defined since all $V_i$ and $V\minus V_i$ are disjoint, open and closed. Then, $g=\sum g_i$ and, in particular, $e_1=\sum \Phi^{-1}(g_i)$.
We obtain that $\Phi^{-1}(g_1)$ is $e_1$ on $U_1\minus X_\sing$ and $0$ on the other irreducible components.
This contradicts the fact that $\Phi^{-1}(g_1)\in\hol^r(U)$ is a holomorphic function on the connected $U$ 
completing the proof.
\end{proof}

When $X$ is normal, a coherent analytic sheaf
$\sS$ is reflexive if and only if $\sS$ is torsion-free and sections of $\sS$ extend across analytic sets of codimension at least 2.
In particular, every section of a reflexive sheaf on the regular part of $X$ extends across the singular set.

\begin{proposition}\label{pro:general_GKKP}
Let $X$ be a normal complex space,  and $\sE$ be a coherent analytic subsheaf of $\tOm_X^1$ such that $\sE_{X_\reg}\iso \Om_{X_\reg}^1$.
If $\sE$ is locally free, then $X$ is smooth.
\end{proposition}

\begin{proof}
We get $\sT_{X_\reg} \iso (\tOm_{X_\reg}^1)^\ast\iso\sE_{X_\reg}^\ast$.
Since $\sT_X$ is reflexive, this isomorphism extends across the singular set and so, $\sT_X\iso\sE^\ast$ is locally free. 
\medskip

With $\sT_X$ and $\sE$ locally free and $\sE\subset\tOm_X^1$, we obtain the smoothness of $X$ in the same way as in the proof of Theorem 6.1 in \cite{GKKP}:

To get a contradiction we assume that $X$ is not smooth. Let $x\in X$ be a singular point of $X$. By shrinking $X$ to a small neighbourhood of $x$, we may assume that $\sT_X$ and $\sE$ are free.
There exists a resolution of singularities, $\rho_X\colon R(X)\rightarrow X$ with simple normal crossing exceptional divisor.
We may assume that $\rho$ is functorial with respect to smooth morphisms (flat submersions%
\footnote{For pure dimensional complex spaces $X$ and $Y$, a holomorphic morphism $f\colon X\rightarrow Y$ is a submersion if the relative $\Om_{X|Y}^1$ is locally free of rank $\dim X-\dim Y$ (which is the case if and only if $f$ can be seen locally as a projection).}%
), see, \eg Theorem 3.45 in \cite{Kollar07}. The functorial property means that if $f\colon X\rightarrow Y$ is a smooth morphism, then we can lift $f$ to the resolutions, \ie there exists a smooth morphism $R(f)\colon R(X)\rightarrow R(Y)$ such that $f\circ \rho_X = \rho_Y\circ R(f)$.
Let $E$ denote the exceptional divisor of $\rho:=\rho_X$ in $R:=R(X)$ which is not empty (since $X$ is not smooth).
Since $\sT_X$ is free, there is a frame of sections $\theta_1{,}...,\theta_n$ which generates $\sT_X$ on $X$.
Since the singular set of $X$ is invariant under any (local) automorphism and $\rho$ is functorial, Corollary 4.7 in \cite{Greb-Kebekus-Kovacs} gives us
	\[\sT_X \iso \rho_\ast \sT_{R}(-\log E)\]
where $\sT_{R}(-\log E)$ denotes the logarithmic tangent sheaf (\ie vector fields which are tangent to $E$ in smooth points of $E$; $\sT_{R}(-\log E)$ is the dual of the logarithmic differential forms $\Om^1_{R}(\log E)$).
Hence, (we are in the special situation that) we may lift the vector fields $\theta_i$ to logarithmic vector fields
	\begin{equation}\label{eq:lifts}\tilde\theta_i\in H^0(R,\sT_{R}(-\log E))\aus H^0({R},\sT_{R}).\end{equation}
Since $\sE\iso\sT_X^\ast$, there is a dual frame $\omega_1{,}...,\omega_n\in \sE\subset\tOm^1_X$ of $\theta_1{,}...,\theta_n$.
Since $\tOm^1_X=\rho_\ast\Om^1_{R}$, $\omega_i$ is given by holomorphic forms $\tilde\omega_1{,}...,\tilde\omega_n\in H^0({R},\Om^1_{R})$ on ${R}$.
On ${R}\minus E$, we have $\tilde\omega_i(\tilde\theta_j)=\omega_i(\theta_j)=\delta_{ij}$.
These equalities extend to $E$ by continuity.
In particular, $\tilde\theta_1(p){,}...,\tilde\theta_n(p)$ are linear independent for every point $p\in E$.
We obtain a contradiction since all $\tilde\theta_i(p)$ are in $T_{p}E$ (because of \eqref{eq:lifts}) with $\dim T_{p}E=n-1$ for $p\in E_\reg$ 
This completes the proof. 
\end{proof}

\begin{proof}[Proof of  Theorem \ref{main}]
As already mentioned, we of course only need to prove that local freeness of $\tOm_X$ or $\alpha_X^1$ implies $X$ is smooth:
Since $\tOm_X^1:=\pi_\ast\Om_M^1$ and $\alpha:=\pi_\ast(\pi^{\ast\ast}\Om_X^1)$,  the assumption implies $X$ is normal by Proposition \ref{normality}. The smoothness of $X$ follows now by Proposition \ref{pro:general_GKKP}.
\end{proof}

\bigskip

Corollary \ref{cor:pinkham} follows from Theorem \ref{main} and the following proposition which is a generalization of a result of Pinkham (see \cite[Proposition~1, Appendix]{Pinkham}).

\begin{proposition}\label{pro:pinkham}
Let $X$ be a reduced Stein space, $\pi\colon M\rightarrow X$ a resolution of singularities with $E$ as exceptional divisor having normal crossings. If  
	\begin{equation}\label{eq:pinkham-cond}
	H_E^1(\Om_M^1(\log E)\otimes\sO(-E))=0,\end{equation}
then
	\[\tOm^1_X= \omega^1_X.\]
\end{proposition}
\medskip

Thereby, $\Om^1_M(\log\! E)$ denotes the logarithmic differential forms.
Pinkham proved that
	\begin{equation}\label{eq:more_general_pinkham}
	H_E^1(\Om_M(\log E))=0=H^0(M,\Om_M(\log E)\otimes\hol_E)
	\end{equation}
implies \eqref{eq:pinkham-cond}.
Additionally, J.~Wahl proved that if $X$ is a surface with rational singularities, then \eqref{eq:pinkham-cond} is satisfied, too (see \cite[Théorème (Wahl)]{Pinkham}, using \cite[\S\-, 4--5]{Wahl}).

The condition \eqref{eq:pinkham-cond} does not imply that $\alpha_X^1\iso \omega_X^1$: By \cite[Section 6.2]{Barlet18}, we have
	\[\Om_{S_k}^1 /\tor=\alpha_{S_k}^1\subsetneq \tOm_{S_k}^1=\omega_{S_k}^1\]
for the surface $S_k:=\{(x,y,z)\colon xy=z^k\}$ in $\CC^3$. Yet, $S_k$ has rational singularities so that \eqref{eq:pinkham-cond} is satisfied for a resolution of singularities of $S_k$ 
(see \cite[Théorème (Wahl)]{Pinkham}).

\newcommand{\Ma}{{M^\ast}}
\begin{proof}[Proof of Proposition \ref{pro:pinkham}] 
To shorten the notation, let us define $\Om^1_M(-E+\log\! E):=\Om^1_M(\log\! E)\otimes\hol(-E)$.

As first step, we will prove that the inclusion $H^0(M,\Om^1_M) \hookrightarrow H^0(\Ma, \Om^1_M)$ induced by the restriction $\cdot|_{\Ma}$ is surjective with $\Ma:=M\minus E$:
The short exact sequence
	\[0\rightarrow \Om^1_M(-E+\log\!E)|_E\rightarrow \Om^1_M(-E+\log\!E) \rightarrow \Om^1_\Ma(-E+\log\!E) \rightarrow 0\]
gives us a long exact sequence of cohomology whose first $H^1$-term vanishes by our assumption \eqref{eq:pinkham-cond}.
Hence, the restriction map
	\[H^0(M,\Om^1_M(-E+\log\!E)) \rightarrow H^0(\Ma,\Om^1_M(-E+\log\!E))\]
 is surjective. Since it is already an inclusion of sets, we get
	\begin{equation}\label{eq:pinkham}H^0(M,\Om^1_M(-E+\log\!E))=H^0(\Ma,\Om^1_M(-E+\log\!E))=H^0(\Ma,\Om^1_M).\end{equation}
Let us consider a point $p$ in $M$ and coordinates $x_1{,}...,x_n$ of $M$ on a neighbourhood of $p$ such that $E=\{x_1{\cdot}...{\cdot}x_\tau=0\}$. Then $\spfrac{dx_1}{x_1}{,}...,\spfrac{dx_\tau}{x_\tau}$ and $dx_{\tau+1}{,}...,dx_n$ give us a basis of $\Om_M^1(\log E)_p$.
Furthermore, $\{x_1{\cdot}..{\,}\hat x_i\, ..{\cdot}x_\tau {dx_i}\}_{ i=1}^{\tau}\cup\{x_1{\cdot}...{\cdot}x_\tau dx_j\}_{j={\tau+1}}^{n}$ is a basis for $\Om^1_M(-E+\log\!E)_p$, in particular $\Om^1_M(-E+\log\!E)_p \aus \Om^1_{M,p}$.
We obtain the following chain of inclusions, which then have to be equalities:
	\[H^0(\Ma,\Om^1_M) \overset{\eqref{eq:pinkham}}= H^0(M,\Om^1_M(-E+\log\!E)) \subset H^0(M,\Om^1_M) \subset  H^0(\Ma,\Om^1_M).\]

\medskip

As second step, we want to use the global extension property proven in the first step to obtain the local extension:
Since $\omega^1_X$ is coherent and $X$ is Stein, Cartan's Theorem A implies that there exists an epimorphism of sheaves
	\[\hol^N_X\twoheadrightarrow \omega^1_X, (f_{i,p})_{i=1}^N\mapsto \sum{\!\Big.}_{i}\, f_{i,p}\cdot \ph_{i,p}\]
(on the whole $X$) given by sections $\ph_1{,}...,\ph_N\in\omega^1_X(X)\aus H^0(\Ma,\Om^1_M)$.
In the first step, we proved that $\ph_i$ can be extended to sections in $H^0(M,\Om^1_M)=H^0(X,\tOm^1_X)$.
If $\eta_p$ is a germ in $\omega^1_{X,p}$, then there exist holomorphic germs $f_{1,p}{,}...,f_{N,p}\in\hol_{X,p}$ such that $\eta_p=\sum_{i=1}^p f_{i,p} \ph_{i,p}$. Since $\ph_{i,p}\in\tOm^1_{X,p}$, we get that $\eta_p$ can be considered as a germ of $\tOm^1_{X,p}$ 
as claimed.
\end{proof}

If $\codim X_\sing$ is at least $2$, then $\omega^1_X =i_\ast \Om^1_{X_\reg}$ \cite[end of Section 2]{Barlet} (where $i\colon X_\reg \hookrightarrow X$). This is not so for $\codim X_\sing=1$  in general.
However, we obtain the following corollary of Proposition \ref{pro:pinkham} by slight modifications of the proof.
\begin{corollary}\label{cor:pinkham_extended}
Let $X$ be a reduced complex space, $\pi\colon M\rightarrow X$ a resolution of singularities with $E$ as exceptional divisor having normal crossings.\\
\stdpt{a} If $X$ is Stein, $i_\ast \Om^1_{X_\reg}$ is coherent, and 
	\begin{equation*}
	H_E^1(\Om_M^1(\log E)\otimes\sO(-E))=0,
	 \hbox{\hspace{.5em} or}\end{equation*}
\stdpt{b} if for all small enough open sets $U\subset X$,
	\begin{equation*}
	H_E^1(\Om_{\pi^{-1}(U)}^1(\log E)\otimes\sO(-E))=0,
	\end{equation*}
then
	\begin{equation}\label{eq:canonical-sheaves-coincide}\tOm^1_X= \omega^1_X =i_\ast \Om^1_{X_\reg}.\end{equation}
\end{corollary}
\begin{proof}
We follow the same arguments as in the proof of  Proposition \ref{pro:pinkham}.
In case the conditions of (a) are satisfied, we obtain in a first step the global extension property:
	\[H^0(X,\tOm^1_X) =H^0(M,\Om^1_M) = H^0(\Ma, \Om^1_M)=H^0(X_\reg, \Om^1_{X_\reg}).\]
Replacing $\omega_X^1$  by $i_\ast \Om^1_{X_\reg}$ in the second step of the proof above, we get the local extension property \eqref{eq:canonical-sheaves-coincide}.

\medskip

Only using the first step of the proof of Proposition \ref{pro:pinkham}, we obtain that the condition $(b)$ implies that the inclusion $H^0(\pi^{-1}(U),\Om^1_{\pi^{-1}(U)}) \hookrightarrow H^0(\pi^{-1}(U)\minus E, \Om^1_{\pi^{-1}(U)})$ is surjective for all small enough open sets $U\subset X$, \ie $\tOm^1_X(U)= i_\ast \Om^1_{X_\reg}(U)$.
This implies that the inclusions $\tOm^1_X\subset \omega^1_X \subset i_\ast \Om^1_{X_\reg}$ are equalities 
as claimed.
\end{proof}

\medskip

We conclude this work with collecting cases  where local freeness of $\omega_X^1$ implies smoothness of $X$. We get the following corollary of Theorem \ref{main} and the above mentioned results.

\begin{corollary}\label{cor:for_the_referee}
Let $X$ be a reduced complex space such that one of the following conditions is satisfied:\\
\stdpt{i} $X$ is normal and $\codim X_\sing\geq 3$. \\
\stdpt{ii} $X$ is normal and has klt singularities.\\
\stdpt{iii} $X$ is a surface with rational singularities, or more general\\
\stdpt{iv} $X$ is Stein and satisfies \eqref{eq:more_general_pinkham}, or more general\\
\stdpt{v} $X$ satisfies the assumptions in Corollary \ref{pro:pinkham}, or\\
\stdpt{vi} $X$ satisfies the conditions (a) or (b) of Corollary \ref{cor:pinkham_extended}.\\
Then $X$ is smooth if $\omega_X^1$ is locally free (or $\sT_X$ is locally free).
\end{corollary}

\bigskip

\providecommand{\bysame}{\leavevmode\hbox to3em{\hrulefill}\thinspace}
 
\end{document}